\date{October 3, 2007}

\documentclass[11pt,leqno]{article}
\usepackage{amssymb,amsmath,amsfonts,amsthm,xypic}
\usepackage{graphicx}
\usepackage{array}

\theoremstyle{plain}
\newtheorem{theorem}{Theorem}[section]

\newtheorem{lemma}[theorem]{Lemma}
\newtheorem{proposition}[theorem]{Proposition}

\theoremstyle{definition}
\newtheorem{example}[theorem]{Example}
\newtheorem{definition}[theorem]{Definition}
\newtheorem{remark}[theorem]{Remark}

\renewcommand{\a}{\alpha}

\newcommand{\PP}{{\mathbb{P}}}
\newcommand{\ZZ}{{\mathbb{Z}}}
\newcommand{\CC}{{\mathbb{C}}}

\newcommand{\RR}{{\mathbb{R}}}

\newcommand{\into}{\longrightarrow}

\newcommand{\inc}{\hookrightarrow}

\newcommand{\la}{\langle}
\newcommand{\ra}{\rangle}
\newcommand{\bd}{\partial}

\newcommand{\del}{\partial}

\newcommand{\bV}{\bigwedge V}

\newcommand{\M}{\mathcal{M}}
\newcommand{\A}{\mathcal{A}}
\newcommand{\B}{\mathcal{B}}

\newcommand{\mc}[1]{\mathcal{#1}}
\newcommand{\IP}[1]{\langle #1\rangle}
\renewcommand{\iff}{if and only if}
\newcommand{\wrt}{with respect to}
\newcommand{\nhood}{neighbourhood}

\begin{document}

\title{Symplectic resolutions, Lefschetz property and formality}

\author{Gil R. Cavalcanti, Marisa Fern\'andez and Vicente Mu\~noz}

\maketitle

\begin{abstract}
We introduce a method to resolve a {\em symplectic orbifold}
$(M,\omega)$ into a smooth symplectic manifold $(\tilde
M,\tilde\omega)$. Then we study how the formality and the
Lefschetz property of $(\tilde M,\tilde\omega)$ are compared with
that of $(M,\omega)$. We also study the formality of the
symplectic blow-up of $(M,\omega)$ along symplectic submanifolds
disjoint from the orbifold singularities. This allows us to
construct the first example of a simply connected compact
symplectic manifold of dimension $8$ which satisfies the Lefschetz
property but is not formal, therefore giving a counter-example to
a conjecture of Babenko and Taimanov.
\end{abstract}


\section{Introduction} \label{sec:introduction}

In \cite{Merk}, Merkulov proved that for a compact symplectic manifold 
the Lefschetz property 
is equivalent to the $d\delta$-lemma, a property similar to the
$dd^c$-lemma for K\"ahler manifolds. Later Babenko and Taimanov
studied formality of symplectic manifolds in \cite{BT2}.  There,
they produced families of non-formal symplectic manifolds in
dimensions strictly greater $8$, all of which failed to satisfy
the Lefschetz property. Due to the fact that the ordinary
$dd^c$-lemma implies formality \cite{DGMS}, they were led to
conjecture that the $d\delta$-lemma (or equivalently the Lefschetz
property) implied formality of symplectic manifolds.

Using symplectic blow up, the first author proved this conjecture
false \cite{Cav2} in all dimensions strictly greater than $2$ and,
for simply connected spaces, in all dimensions strictly greater
than $8$. Further, due to a well known result of Miller and
Neisendorfer \cite{Mi78}, any simply connected manifold of
dimension $6$ or less is formal. Hence the only case where the
conjecture still stood was for simply connected symplectic
$8$-manifolds. As Miller's result suggests, the requirements that
the manifold is simply connected and $8$-dimensional are strong
contraints. Indeed, only recently, in \cite{FM4}, were the first
examples of non-formal simply connected symplectic $8$-manifolds
produced. Now we prove that the conjecture does not hold in $8$
dimensions either, therefore completing the study of the
relationship between the Lefschetz property and formality.

To show that there is no relation between those properties, we
construct an example by merging  and improving  on techniques from
\cite{Cav2} and \cite{FM4}. The tool we use to detect
non-formality is not Massey products, but a new product which
depends on an even cohomology class $a$ and which we call
$a$-Massey product. The method for construction of new symplectic
manifolds is the {\it symplectic resolution of singularities}, in
the spirit of \cite{FM4}  as well as symplectic blow-up. Putting
these together, we study in detail how the $a$-Massey products
products and the Lefschetz property behave under symplectic
blow-up and under symplectic resolution of singularities.

The way we construct our example consists in taking a quotient of
a non-formal symplectic manifold by a (non free) action of a
finite group so that the resulting manifold is a symplectic
orbifold with nontrivial $a$-Massey product. Then we blow up this
orbifold along suitable submanifolds to produce a non-formal
orbifold which satisfies the Lefschetz property and finally we
resolve the isolated symplectic orbifold singularities. The
resulting smooth manifold is a counter-example to the
Babenko--Taimanov conjecture.

This paper is organized as follows. In Section \ref{sec:formality}
we introduce new obstructions to formality called $a$-Massey
products and study their properties. There we also study formality
of orbifolds and show that the minimal model for the topological
space underlying an orbifold is given by the minimal model for the
algebra of orbifold differential forms. Therefore, similarly to
the case of manifolds, in order to check formality of an orbifold
one can simply work with differential forms, instead of piecewise
linear forms on some triangulation.

In Section \ref{sec:resolutions}, we introduce the concept of a
symplectic resolution and show that any symplectic orbifold with
isolated singularities can be resolved into a smooth symplectic
manifold. Our method of resolution of singularities of symplectic
orbifolds works in more cases than that of \cite{NiPas}. Then, we
study the behaviour of $a$-Massey products and the Lefschetz
property under resolutions. We show that both are preserved by
resolution of orbifold singularities.

In Section \ref{sec:blow-up}, we recall results about the
behaviour of the Lefschet property under symplectic blow-up and
give conditions for $a$-Massey products to be preserved under
blow-up. Finally, in the last section we put these ingredients
together to produce the counter-example to  the Babenko--Taimanov
conjecture.

\medskip

\noindent {\bf Acknowledgments.} This work has been partially
supported through grants MCyT (Spain) MTM2004-07090-C03-01 and
MTM2005-08757-C04-02 as well as EPSRC (UK).

\medskip
\section{Formality and $a$-Massey products} \label{sec:formality}

\subsection{Formality of differential graded algebras}

In this section we review the notion of formality \cite{Sul,DGMS}
and Massey products, which are well known  obstructions
to formality.
Then we introduce a new product which depends on an even
cohomology class and is similar to Massey products. This new
product also provides obstructions to formality, much in the
spirit of Massey products, but in some situations they are simpler to compute than
higher order Massey products. We finish with
some comments about
the formality of manifolds and orbifolds.

We work with differential graded commutative algebras, or DGAs,
over the field  of real numbers, $\RR$. We denote the degree of an
element $a$ of a DGA by $|a|$. A DGA $(\A,d)$ is {\it minimal\/}
if:
\begin{enumerate}
 \item $\A$ is free as an algebra, that is, $\A$ is the free
 algebra $\bigwedge V$ over a graded vector space $V=\oplus V^i$, and
 \item there exists a collection of generators $\{ a_\tau,
 \tau\in I\}$, for some well ordered index set $I$, such that
 $|a_\mu|\leq |a_\tau|$ if $\mu < \tau$ and each $d
 a_\tau$ is expressed in terms of preceding $a_\mu$ ($\mu<\tau$).
 This implies that $da_\tau$ does not have a linear part, i.e., it
 lives in $\bV^{>0} \cdot \bV^{>0} \subset \bV$.
\end{enumerate}

Given a differential algebra $(\A,d)$, we denote its cohomology by
$H(\A)$. The cohomology of a differential graded algebra $H(\A)$
is naturally a DGA with the product induced by that on $\A$ and
with differential identically zero. The DGA $\A$ is {\it
connected} if $H^0(\A)=\RR$.

A differential algebra $({\cal M},d)$ is a {\it minimal model} of
$(\A,d)$ if $({\cal M},d)$ is minimal and there exists a morphism
of differential graded algebras $\rho\colon {({\cal
M},d)}\longrightarrow {(\A,d)}$ inducing an isomorphism
$\rho^*\colon H({\cal M})\longrightarrow H(\A)$ in cohomology.
In~\cite{H} Halperin proved that any connected differential
algebra $(\A,d)$ has a minimal model unique up to isomorphism.

A DGA  ${\cal A}$ with minimal model $\M$ is {\it formal} if there
is a morphism of differential algebras $\psi\colon {{\cal
M}}\longrightarrow H({\cal A})$ which induces an isomorphism in
cohomology. In this case $\M$ is simultaneously the minimal model
for $\A$ and $H(\A)$.

In order to detect non-formality, instead of computing the minimal
model, which usually is a lengthy process, we can use Massey
products, which are obstructions to formality. The simplest type
of Massey product is the triple (also known as ordinary) Massey
product, which we define next.

Let $\A$ be a DGA  and $a_i \in \A$, $1 \leq i\leq 3$, be three
closed elements  such that $a_1\wedge a_2$ and $a_2\wedge a_3$ are
exact. The {\it (triple) Massey product} of the $a_i$ is the set
  $$
  \langle a_1,a_2,a_3 \rangle  = \{
  [ a_1 \wedge a_{2,3}+(-1)^{|a_1|+1} a_{1,2}
  \wedge a_3] \ |\, da_{1,2} = a_1\wedge a_2,
  \ da_{2,3} = a_2\wedge a_3\} \subset
  H^{|a_1|+|a_2|+ |a_3| -1}(\A),
  $$
where $|a_i|$ is the degree of $a_i$. This set depends only on the
cohomology classes of the $a_i$ and not on the $a_i$ themselves,
hence this expression  also defines a product for cohomology
classes\footnote{In the literature it is also usual to call the
induced product in cohomology Massey product. This difference is
purely semantic and does not change any of the arguments used in
the paper.}. Given  $a_{1,2}$ and $a_{2,3}$  as above, we can add
any closed elements $\alpha_{1,2}$ and $\alpha_{2,3}$  to them and
we still have  the equalities
 $$
 d(a_{1,2} +  \alpha_{1,2}) = a_1 \wedge a_2 \qquad \mbox{ and }
 \qquad  d(a_{2,3} +  \alpha_{2,3}) = a_2 \wedge a_3,
 $$
hence we see that $\langle a_1,a_2,a_3 \rangle$ is a set of the form $c+ ([a_1]
\wedge H^{|a_2|+|a_3|-1}(\A) + H^{|a_1|+|a_2|-1}(\A)\wedge
[a_3])$. So the Massey product  gives a well-defined element in
  $$
  \frac{H^{|a_1|+|a_2|+|a_3|-1}(\A)}{[a_1]
  \wedge H^{|a_2|+|a_3|-1}(\A) + H^{|a_1|+|a_2|-1}(\A)\wedge [a_3]} \; .
  $$
We say that $\langle a_1,a_2,a_3 \rangle$ is {\it trivial} if
$0\in \langle a_1,a_2,a_3 \rangle$. The {\it indeterminacy} of the
Massey product is the set
 $$
 \{c-c'|\,c,c' \in \la a_1,a_2,a_3\ra\}=[a_1]
 \wedge H^{|a_2|+|a_3|-1}(\A) + H^{|a_1|+|a_2|-1}(\A)\wedge [a_3].
  $$

Now we move on to the definition of higher Massey products
(see~\cite{TO}). Given $a_{i}\in \A$, $1\leq i\leq n$, $n\geq 3$,
the Massey product $\la a_{1},a_{2},\ldots,a_{n}\ra$, is defined
if there are elements  $a_{i,j}$ on $\A$, with $1\leq i\leq j\leq
n$, except for the case $(i,j)=(1,n)$, such that
 \begin{equation}\label{eqn:gm}
 \begin{aligned}
 a_{i,i}&= a_i,\\
 d\,a_{i,j}&= \sum\limits_{k=i}^{j-1} \overline{a_{i,k}}\wedge
 a_{k+1,j},
 \end{aligned}
 \end{equation}
where $\bar a=(-1)^{|a|} a$. Then the {\it Massey product} is
 $$
 \la a_{1},a_{2},\ldots,a_{n} \ra =\left\{
 \left[\sum\limits_{k=1}^{n-1} \overline{a_{1,k}} \wedge
 a_{k+1,n}\right] \ | \ a_{i,j} \mbox{ as in (\ref{eqn:gm})}\right\}
 \subset H^{|a_{1}|+ \ldots +|a_{n}|
 -(n-2)}(\A)\, .
 $$
We say that the Massey product is {\it trivial} if $0\in \la
a_{1},a_{2},\ldots,a_{n,}\ra$. Note that for $\la
a_{1},a_{2},\ldots,a_{n}\ra$ to be defined it is necessary that
the lower order Massey products $\la a_{1},\ldots,a_{i}\ra$ and
$\la a_{i+1},\ldots,a_{n}\ra$ with $2 < i < n-2$ are defined and
trivial. As before, the indeterminacy of the Massey product is
 $$
 \{c - c'|\,c,c' \in  \la a_{1},a_{2},\ldots,a_{n}\ra\}.
 $$
However, in contrast with the triple products, in general there is
no simple description of this set.

The relevance of Massey products to formality comes from the
following well known result.

\begin{theorem}[\cite{DGMS,TO}] \label{theo:criterio1}
A DGA which has a non-trivial Massey product  is not formal.
\end{theorem}

\subsection{$a$-Massey products}

Next, we introduce another obstruction to the formality, which
generalizes the triple Massey products and  has the advantage of
being simpler for computations than the higher order Massey
products.

\begin{proposition}\label{prop:GMassey}
Let $\A$ be a DGA and  let $a, b_1, \ldots, b_n \in \A$ be
closed elements such that the degree $|a|$ of $a$ is even  and $a\wedge b_i$ is
exact, for all $i$. Let $\xi_i$ be any form such that $d\xi_i = a
\wedge b_i$. Then the form
 \begin{equation}\label{eq:c}
 c= \sum_i \overline{\xi_1}\wedge\ldots \wedge \overline{\xi_{i-1}}
 \wedge b_i \wedge \xi_{i+1} \wedge \ldots \wedge \xi_n
 \end{equation}
is closed, where  $\overline{\xi} = (-1)^{|\xi|}\xi$.
\end{proposition}

\begin{definition}\label{def:GMassey}
In the situation above, the {\it $n^{th}$ order $a$-Massey
product} of the $b_i$ (or just {\it $a$-product}) is the subset
 $$
\la a; b_1,\ldots,b_n \ra :=  \left\{ \left[\sum_i
\overline{\xi_1}\wedge\ldots \wedge\overline{\xi_{i-1}} \wedge b_i
\wedge \xi_{i+1} \wedge \ldots \wedge \xi_n\right] \, |\,
  d\xi_i = a\wedge b_i\right\}\subset  H(\A).
  $$
We say that the $a$-Massey product is {\it trivial} if $0\in \la
a; b_1,\ldots,b_n\ra$.
\end{definition}

If $n=2$, the product introduced above is just the triple Massey
product $\IP{b_1,a,b_2}$, but for higher values of $n$ these
products are different to the higher order Massey products. In the
applications we will use the $3^{rd}$ order $a$-product with $b_i$
even degree elements, so that the product can be written as
 \begin{equation}\label{eq:triple product}
 \IP{a;b_1,b_2,b_3}= \{[b_1\wedge  \xi_2 \wedge \xi_3 + c.p.]\, |\, d\xi_i = a \wedge b_i\},
 \end{equation}
where $c.p.$ stands for cyclic permutations. The product
\eqref{eq:triple product} appeared before in \cite{FM4} in the
same context we will use it later.

Now we study the indeterminacy of this product and  show that the
$a$-product is an obstruction to formality.

\begin{lemma}\label{lem:symmetry}
Let $\sigma$ be the permutation of $\{1, \ldots, n\}$ which is
just the transposition of $j$ and $j+1$, for some $j$. Then given
$a$, $b_i$ and $\xi_i$ as above, we have $c =
(-1)^{(|b_j|+1)(|b_{j+1}|+1)}c_{\sigma}$, where $c$ is given by
equation \eqref{eq:c} and
 $$
 c_{\sigma} = \sum_i \overline{\xi_{\sigma(1)}}\wedge\ldots
 \wedge \overline{\xi_{\sigma(i-1)}} \wedge b_{\sigma(i)} \wedge
 \xi_{\sigma(i+1)} \wedge \ldots \wedge \xi_{\sigma(n)}.
 $$
\end{lemma}
The proof is a straightforward computation.

\begin{lemma}\label{lem:cohomology a b}
The $a$-Massey product $\la a;b_1,\ldots,b_n\ra$ only depends on the
cohomology classes $[a]$, $[b_i]$ and not on the particular
elements representing these classes.
\end{lemma}

\begin{proof}
Let $ a + d \alpha$ be another representative for the class $[a]$.
Then, the generic element in $\IP{a+d\alpha;b_1, \ldots,b_n}$ is
given by the cohomology class of
 \begin{equation}\label{eq:a + dalpha}
 c' = \sum_i \overline{\xi_1'} \wedge \ldots
 \wedge \overline{\xi_{i-1}'}\wedge b_i\wedge \xi_{i+1}'\wedge \ldots \wedge \xi_n'.
 \end{equation}
with $\xi_i'$ satisfying $d\xi_i' = (a + d\alpha) \wedge b_i$.
Thus, $\xi_i = \xi_i' - \alpha \wedge b_i$ satisfies $d\xi_i
= a \wedge b_i$. Since $a$ is of even degree, $\alpha$ is of odd
degree and hence $\alpha^2 =0$. Therefore, letting $c$ be given by
equation \eqref{eq:c} we have
 \begin{align*}
 c' =&\,  \sum_i (\overline{\xi_1 +\alpha \wedge b_1}) \wedge \ldots
 \wedge b_i\wedge \ldots \wedge  (\xi_n +\alpha \wedge b_n) \\
 =&\,  c + \sum_{j<i} \overline{\xi_1}\wedge \ldots \wedge
 \overline{\xi_{j-1}}\wedge \overline{\alpha \wedge b_j}
 \wedge \overline{\xi_{j+1}}\wedge \ldots \wedge
 \overline{\xi_{i-1}}\wedge b_i\wedge \ldots \wedge \xi_n \\
 &+ \sum_{i<j} \overline{\xi_1}\wedge\ldots  \wedge
 \overline{\xi_{i-1}}\wedge b_i\wedge\xi_{i+1} \wedge
 \ldots \wedge \alpha \wedge b_j \wedge  \ldots \wedge \xi_n  \\
 = &\, c + \sum_{j<i} \overline{\xi_1}\wedge \ldots \wedge
 \overline{\xi_{j-1}}\wedge \overline{\alpha \wedge b_j} \wedge
 \overline{\xi_{j+1}}\wedge \ldots \wedge
 \overline{\xi_{i-1}}\wedge b_i\wedge \ldots \wedge \xi_n \\&
 -\sum_{i<j} \overline{\xi_1}\wedge\ldots \wedge \overline{\xi_{i-1}}
 \wedge \overline{\alpha\wedge b_i}\wedge \overline{\xi_{i+1}}\wedge
 \ldots \wedge \overline{\xi_{j-1}}  \wedge b_j \wedge  \ldots \wedge \xi_n \\
 = &\,  c,
\end{align*}
where in the second equality we have expanded the expression for
$c'$ and used $\alpha^2 =0$, in the third equality we used that
$\eta \wedge \alpha = - \overline{\alpha} \wedge \overline{\eta}$
for any form $\eta$, as $\alpha$ is odd, and in the last we used
that the two sums are the same, with the roles of $i$ and $j$
reversed. This shows that the $a$-Massey product only depends on $[a]$.

A similar computation shows that the same is true for the $b_i$.
We do it for $i=1$. If $b_1+d\alpha$ is another representative for
the class $[b_1]$, then the generic element in
$\IP{a;b_1+d\alpha,b_2, \ldots,b_n}$ is given by the cohomology
class (\ref{eq:a + dalpha}), where $\xi_i'$ satisfies $d\xi_1' = a
\wedge (b_1+d\alpha)$, and $d\xi_i'=a\wedge b_i$ for $i>1$. Take
$\xi_1=\xi_1'-a\wedge \alpha$ and $\xi_i'=\xi_i$ for $i>1$, so
that $d\xi_i = a \wedge b_i$. Letting $c$ be given by equation
\eqref{eq:c} we have
 \begin{align*}
 c' =&\, (b_1+d\alpha) \wedge \xi_2 \wedge \ldots \wedge  \xi_n +
 \sum_{i>1} (\overline{\xi_1 + a\wedge \alpha}) \wedge \overline{\xi_2} \wedge \ldots
 \wedge b_i\wedge \ldots \wedge  \xi_n \\
 =&\,  c + (d\alpha \wedge \xi_2 \wedge \ldots \wedge  \xi_n +
 \sum_{i>1} \overline{a\wedge \alpha} \wedge \overline{\xi_2} \wedge \ldots
 \wedge b_i\wedge \ldots \wedge  \xi_n) \\
 =&\,  c + d( \alpha \wedge \xi_2 \wedge \ldots \wedge  \xi_n )\,
 ,
\end{align*}
where we have used that $a$ is of even degree.
\end{proof}

One computation relevant to the $a$-Massey products consists in checking
what happens to them when one changes the $\xi_i$ by closed forms
$\eta_i$, as this gives the indeterminacy of this product.

\begin{lemma}\label{lem:helps in the computations}
Fix $j \in \{1,\ldots,n\}$ and let $\xi_j' =  \xi_j +\eta_j$ for
some closed element $\eta_j$. Let $c$ be given by equation
\eqref{eq:c} and $c'$ be given by the same equation but with
$\xi_j$ swapped by $\xi_j'$. Then
 \begin{equation}
 \begin{aligned}\label{eq:indeterminacy}
 c'&= c+ (-1)^{(|b_j|+1)(n-j+\sum_{i>j}|b_i|)}\left(\sum_{i=1, i\neq j}^n
 \overline{\xi_1}\wedge\ldots\wedge \widehat{ \overline{\xi_j}}\wedge \ldots
 \wedge \overline{\xi_{i-1}} \wedge b_i \wedge \xi_{i+1} \wedge \ldots \wedge
 \xi_n\right) \wedge \eta_j,
 \end{aligned}
 \end{equation}
where and $\widehat{ \xi_j}$ indicates that the term $\xi_j$ is
skipped in the product.
 \end{lemma}

\begin{proof}
Using Lemma \ref{lem:symmetry} and Proposition \ref{prop:GMassey},
we have
\begin{align*}
c' &  = (-1)^{(|b_j|+1)(n-j+\sum_{i>j}|b_i|)}\sum_{i=1}^n
\overline{\xi_1}\wedge\cdots\wedge \widehat{
\overline{\xi_j}}\wedge \cdots \wedge \overline{\xi_{i-1}} \wedge
b_i \wedge \xi_{i+1} \wedge \cdots \wedge \xi_n \wedge
(\xi_j+\eta_j)\\
& =(-1)^{(|b_j|+1)(n-j+\sum_{i>j}|b_i|)}\sum_{i=1}^n
\overline{\xi_1}\wedge\cdots\wedge \widehat{
\overline{\xi_j}}\wedge \cdots \wedge \overline{\xi_{i-1}} \wedge
b_i \wedge \xi_{i+1} \wedge \cdots \wedge \xi_n \wedge \xi_j\\
&+ (-1)^{(|b_j|+1)(n-j+\sum_{i>j}|b_i|)}\sum_{i=1, i \neq j}^n
\overline{\xi_1}\wedge\cdots\wedge \widehat{
\overline{\xi_j}}\wedge \cdots \wedge \overline{\xi_{i-1}} \wedge
b_i \wedge \xi_{i+1} \wedge \cdots \wedge \xi_n \wedge \eta_j\\ &
= c +  (-1)^{(|b_j|+1)(n-j+\sum_{i>j}|b_i|)}\left(\sum_{i=1,i \neq
j}^n \overline{\xi_1}\wedge\cdots\wedge \widehat{
\overline{\xi_j}}\wedge \cdots \wedge \overline{\xi_{i-1}} \wedge
b_i \wedge \xi_{i+1} \wedge \cdots \wedge \xi_n \right) \wedge
\eta_j.
\end{align*}
\end{proof}

Observe that up to a sign, the coefficient of $\eta_j$ in
\eqref{eq:indeterminacy} is an element in
$\IP{a;b_1,\ldots,\widehat{b_j},\ldots,b_n}$, hence Lemma
\ref{lem:helps in the computations} proves the following inductive
way to compute the indeterminacy of the $a$-product.

\begin{proposition}\label{prop:indet}
The indeterminacy of the $a$-product $\IP{a;b_1,\ldots,b_n}$  is a
subset of
 $$
 \sum_{j=1}^n \IP{a;b_1,\ldots,\widehat{b_j},\ldots,b_n}\wedge H(\A).
 $$
In particular, the indeterminacy of the triple $a$-product $\IP{a;b_1,b_2,b_3}$ is a subset of
 $$
 \IP{b_1,a,b_2} \wedge H(\A)+ \IP{b_2,a,b_3}\wedge H(\A)+ \IP{b_3,a,b_1}\wedge H(\A),
 $$
where $\IP{\bullet,\bullet,\bullet}$ is the
(ordinary) triple Massey product.
\end{proposition}

\begin{remark}\label{rem:eta exact}
As a corollary to Lemma \ref{lem:helps in the computations} we see
that if we change $\xi_i$ by an exact form, the cohomology class
of the representative of the product does not change. Together
with Lemma \ref{lem:cohomology a b}, this tells us that in order
to compute the $a$-Massey product, one does not have to worry
about the particular forms $a$ and $b_i$ chosen to represent their
cohomology classes. Further, once we fix one choice of $\xi_i$, in
order to obtain any other element in the set $ \la a;b_1,\ldots
,b_n\ra$ we only have to pick one representative for each
cohomology class of degree $|\xi_i|$ and add that to $\xi_i$.
\end{remark}

Next we show that the $a$-Massey products are well behaved under
quasi-isomorphisms.

\begin{lemma}\label{lem:quasi isomorphisms}
Let  $\psi:\mc{B} \into  \A$ be a quasi isomorphism. If $\A$ has
an $a$-Massey product, say $\IP{a;b_1,\ldots, b_n}$, and $a', b_i' \in
\mc{B}$ are such that $[\psi(a')] = [a]$ and $[\psi(b_i')] =
[b_i]$ then, the $a'$-Massey product $\IP{a';b_1',\ldots,b_n'}$ is
defined and satifies
 \begin{equation}\label{eq:quasi iso}
 \psi(\IP{a';b_1',\ldots , b_n'}) = \la
 a;b_1,\ldots,b_n\ra.
 \end{equation}
Conversely, if $\IP{a';b_1',\ldots,  b_n'} $ is an $a'$-Massey product on
$\mc{B}$ and $a= \psi(a')$ and $b_i = \psi(b_i')$ then the
identity \eqref{eq:quasi iso} also holds.
\end{lemma}

\begin{proof}
We will only prove the first claim as the second is analogous.
First, since $\psi$ is a quasi-isomorphism, there are $a'$ and
$b_i' \in \mc{B}$ such that $[a] = [\psi(a')]$ and $[b_i]=
[\psi(b_i')]$.  Further, again because $\psi$ is a
quasi-isomorphism $a' \wedge b_i'$ is exact on $\mc{B}$, say
$a'\wedge b_i' = d\xi_i'$, so the $a'$-product
$\IP{a';b_1',\ldots,b_n'}$ is defined in $\mc{B}$.

Now we prove \eqref{eq:quasi iso}. We start showing that
$$\psi(\IP{a';b_1',\ldots , b_n'}) \supset \la a;b_1,\ldots,b_n\ra.$$
Let $[c]\in\la a;b_1,\ldots, b_n\ra$ be an element in the
$a$-product in $\mc{\A}$. According to Lemma \ref{lem:cohomology a
b},
  $$
  [c]\in\la \psi(a');\psi(b_1'),\ldots, \psi(b_n')\ra= \la a;b_1,\ldots, b_n\ra.
  $$
Therefore we may write $[c]=  [\sum \overline{\xi_1}\wedge
\ldots\wedge \psi(b_i')\wedge\ldots  \wedge \xi_n]$ with
$d\xi_i=\psi(a')\wedge\psi(b_i')$, hence $d(\xi_i-\psi(\xi_i'))=0$
and  $\xi_i-\psi(\xi_i')$ are closed elements in $\A$ which
represents cohomology classes in $H(\mc{A}) \cong H(\mc{B})$.
Hence there are elements $\zeta_i\in \mc{B}$  and  $z_i\in \A$
such that $\zeta_i$ are closed and $\psi(\xi_i' + \zeta_i) = \xi_i
+ dz_i$ . So according to Lemma \ref{lem:helps in the
computations} and Remark \ref{rem:eta exact}, the class
  $$
  [c'] = [\sum  (\overline{\xi_1'+\zeta_1}) \wedge \ldots \wedge b_i'\wedge
  \ldots \wedge (\xi_n' + \zeta_n)] \in \IP{a';b_1'\ldots,b_n'}
  $$
  satisfies
 \begin{align*}
 \psi([c'])&= [\sum \overline{\psi(\xi_1'+\zeta_1)} \wedge
 \ldots \wedge \psi(b_i') \wedge \ldots \wedge \psi(\xi_n' + \zeta_n)]  \\
 & =[\sum (\overline{\xi_1+d z_1}) \wedge \ldots \wedge \psi(b_i')\wedge
  \ldots \wedge (\xi_n + d z_n)]  \\
 &=[\sum \overline{\xi_1}\wedge \ldots \wedge \psi(b_i')\wedge \ldots  \wedge \xi_n]=[c],
 \end{align*}
which proves the inclusion.

To prove the other inclusion, let $[c']\in \la
a';{b}_1',\ldots,{b}_n'\ra$, and write $[c']=  [\sum
\overline{\xi_1'}\wedge \ldots \wedge b_i'\wedge \ldots  \wedge
\xi_n']$ with $d\xi_i'=a'\wedge b_i'$. Applying $\psi$ to this
expression we see that
 \begin{align*}
 \psi([c']) &=  [\sum \psi(\overline{\xi_1'})\wedge \ldots \wedge
 \psi(b_i')\wedge \ldots  \wedge \psi(\xi_n')] \in
 \IP{\psi(a');\psi(b_1'),\ldots,\psi(b_n')} = \IP{a;b_1,\ldots,b_n},
 \end{align*}
as we wanted.
\end{proof}

The obvious implication of this lemma is that $a$-Massey products
are obstructions to formality.

\begin{theorem}\label{prop:non-formal}\label{theo:amassey and formality}
If a DGA  has a nontrivial $a$-Massey product, then it is not
formal
\end{theorem}

\begin{proof}
Indeed, if $\A$ has a nontrivial $a$-Massey product then, according to
Lemma \ref{lem:quasi isomorphisms}, so does its minimal model. On
the other hand, $H(\A)$ never has a nontrivial product, so the
minimal models for $\A$ and $H(\A)$ can not be the same.
\end{proof}

Actually, the $a$-Massey product of degree 2 elements is the
first obstruction to formality that appears as an obstruction to
$3$--formality \cite{FM1} for a simply connected manifold, and
which is different from a Massey product.

\subsection{Formality of manifolds and orbifolds}

The {\it minimal model\/} $\M$ of a connected differentiable
manifold $M$ is the minimal model for the de Rham complex
$(\Omega(M),d)$ of differential forms on $M$. If $M$ is simply
connected, then the dual of the real homotopy vector space
$\pi_i(M)\otimes \RR$ is isomorphic to the space of generators of
$\M$ in degree $i$ for any $i$. This relation also happens when
$i>1$ and $M$ is nilpotent, that is, the fundamental group
$\pi_1(M)$ is nilpotent and its action on $\pi_j(M)$ is nilpotent
for $j>1$ (see~\cite{DGMS}).

A manifold $M$ is {\it formal} if  $(\Omega(M),d)$ is formal.
Therefore, if $M$ is formal and simply connected, then the real
homotopy groups $\pi_i(M)\otimes \RR$ are obtained from the
minimal model of $H(M)$.

Many examples of formal manifolds are known: all compact symmetric
spaces (e.g., spheres, projective spaces, compact Lie groups, flag
manifolds),  compact K\"ahler manifolds and simply connected
manifolds of dimension six or less. The importance of formality in
symplectic geometry stems from the fact that it allows to
distinguish between symplectic manifolds which admit K\"ahler
structures and some which do not \cite{DGMS,TO}.

Now we extend the definition of formality to orbifolds. Let us
first introduce this concept.

\begin{definition} \label{def:orbifold}
An {\it orbifold} is a (Hausdorff, paracompact) topological space
$M$ with an atlas with charts modelled on $U/G_p$, where $U$ is an
open set of $\RR^n$ and $G_p$ is a finite group acting linearly on
$U$ with only one fixed point $p\in U$. The number $n$ is the {\it
dimension} of the orbifold.
\end{definition}

Note that our definition of orbifold is more restricted to some
other definitions in the literature (e.g.\ \cite{Th2}).

An orbifold $M$ contains a discrete set $\Delta$ of points $p\in
M$ for which $G_p\neq Id$. The complement $M \setminus \Delta$ has
the structure of a smooth manifold. The points of $\Delta$ are
called singular points of $M$. For any singular point $p\in
\Delta$, let $B/G_p$ be a small neighbourhood of $p$, where $B$ is
a ball in $\RR^n$. Then $B/G_p$ is a rational homology ball
(actually it is contractible), and $\bd B/G_p$ is a rational
homology $(n-1)$-sphere.

The space $\Omega^k_{orb}(M)$ of {\em orbifold differential forms}
consists of $k$-forms such that in each chart $U/G_p$ are
$G_p$-invariant elements of $\Omega^k(U)$. Note that there is a
well-defined differential $d$, so that $(\Omega_{orb}(M),d)$ is a
DGA. Its cohomology is called the \textit{orbifold de Rham
cohomology of $M$}.

\begin{definition} \label{def:orbifold-formality}
 Let $M$ be an orbifold. Then a {\it minimal model} for $M$ is a minimal
 model for the DGA $(\Omega_{orb}(M),d)$. The orbifold $M$ is
 {\it formal} if its minimal model is formal.
\end{definition}

\begin{proposition} \label{prop:model-orbifold}
  Let $(\M,d)$ be the minimal model of an orbifold $M$. Then
  $H(\M)=H^*(M)$, where the latter means singular cohomology
  with real coefficients.
\end{proposition}

\begin{proof}
  Let $(\A,d)$ denote the sub-algebra of $(\Omega_{orb}(M),d)$
  such that $\A^0$ consists of functions which are constant on a
  neighbourhood of each singular point, and $\A^k$ consists of
  $k$-forms which are zero on a neighbourhood of each singular
  point. Let us see that
   $$
   (\A,d)\inc (\Omega_{orb}(M),d)
   $$
  is a quasi-isomorphism.

  Let $\a\in\Omega_{orb}(M)$ be closed. Let $p\in \Delta$ and
  consider a neighbourhood $U_p$ of the form $B/G_p$, for $B\subset \RR^n$
  a ball. Then we may
  consider $\a$ as a closed form on $B$. Hence $\alpha$ is
  exact, that is, there exists a form $\beta$ such that
  $\alpha=d\beta$. By averaging by $G_p$ we may assume that
  $\beta$ is $G_p$-equivariant, i.e., $\beta\in\Omega_{orb}(U_p)$.
  Consider a bump function $\rho$ which is zero off $U_p$ and $1$ in
  a smaller neighbourhood of $p$. Then
  $\a-d(\rho\beta)$ is in $\A$ and it is cohomologous to $\a$.
  This proves surjectivity of $H(\A)\to H(\Omega_{orb}(M))$.

  Now suppose that $\a\in \A$ satisfies that $\a=d\beta$, with $\beta\in
  \Omega_{orb}(M)$. Let $V_p=B/G_p$ be a neighbourhood of each $p\in \Delta$, small
  enough so that they are disjoint with the support of $\a$.
  Consider a map $\phi:M\to M$, which
  is the identity off $\bigcup V_p$,
  sending $V_p$ into $V_p$ in such a way that
  it contracts a smaller neighbourhood of $p$ into $p$. We can take $\phi$
  orbi-smooth (that is, it has a $G_p$-equivariant lifting to a map
  $B\to B$ which
  is smooth). So there is a DGA morphism $\phi^*:\Omega_{orb}(M)\to
  \Omega_{orb}(M)$. Then $\a=\phi^*\a =d (\phi^*\beta)$ and
  $\phi^*\beta\in \A$. This proves injectivity of $H(\A)\to H(\Omega_{orb}(M))$.

\bigskip

 With the above at hand, now fix $U_p=B/G_p$  small neighbourhoods of
  $p\in\Delta$. Let $U=\bigcup_{p\in \Delta} U_p$.
  Restriction gives a map $(\A,d)\to (\Omega(M \setminus \bar U),d)$ with kernel
  the DGA $(\B,d)$ consisting of forms which vanish on $M\setminus \bar U$
  and also vanish for positive degrees and are locally constant for degree $0$,
  on a neighbourhood of $\Delta$. Clearly $(\B,d)=\bigoplus_{p\in
  \Delta} (\B_p,d)$, where $\B_p$ consists of forms on $U_p$ which
  vanish on the boundary (so that they can be extended by zero off
  $U_p$) and vanish for positive degrees and are constanta for degree
  $0$,
  on a neighbourhood of $p$. Hence there is a
  exact
  sequence
   $$
   0\to \bigoplus_{p\in
   \Delta} (\B_p,d) \to (\A,d) \to (\Omega(M\setminus \bar U),d) \to
   0\, .
   $$

Working as above, $\B_p\inc \Omega_{orb,0}(B/G_p)$ is a
quasi-isomorphism, where $\Omega_{orb,0}(B/G_p)$ are the orbifold
forms on $B/G_p$ vanishing on the boundary. Clearly,
$\Omega_{orb,0}(B/G_p)=\Omega_0(B)^{G_p}$ is the $G_p$-invariant
part of the forms on $B$ vanishing on the boundary. Thus
 $$
 H^k(\B_p) = H^k(\Omega_0(B))^{G_p}= H^k(B,\bd B)^{G_p}=\left\{
 \begin{array}{ll} \RR, \qquad & k=2n, \\
 0, &\text{otherwise.}
 \end{array} \right.
 $$
This gives an exact sequence
 $$
  H^{k-1}(M\setminus\bar U) \to \bigoplus_{p\in \Delta} H^k(B,\bd B) \to H^k(\A) \to
  H^k(M\setminus\bar U).
  $$
 Together with the exact sequence for singular cohomology
 $$
  H^{k-1}(M\setminus\bar U) \to \bigoplus_{p\in \Delta} H^k(B,\bd B)=H^k(M,M\setminus\bar U)
  \to H^k(M) \to H^k(M\setminus\bar U)\, ,
  $$
which shows the desired result.
\end{proof}

\begin{remark}
The concept of formality is already defined for nilpotent
CW-complexes \cite{GM}. In the case that $M$ is an orbifold whose
underlying space is nilpotent, Definition
\ref{def:orbifold-formality} of formality for $M$ agrees with that
in \cite{GM}. For this it is enough to see that if
$(\Omega_{PL}(M),d)$ denotes the complex of piecewise polynomial
differential forms (for a suitable triangulation of $M$), then the
inclusion $(\Omega_{orb}(M),d) \inc (\Omega_{PL}(M),d)$ gives a
quasi-isomorphism.
\end{remark}

\section{Symplectic resolutions}\label{sec:resolutions}

\subsection{Symplectic orbifolds and their resolutions}

Now we introduce the concepts of symplectic orbifold and
symplectic resolution and show that any symplectic orbifold can be
resolved into a smooth symplectic manifold.

\begin{definition} \label{def:symplectic_orbifold}
  A {\it symplectic orbifold} $(M,\omega)$ is a $2n$-dimensional orbifold $M$ together
  with a $2$-form $\omega\in \Omega^2_{orb}(M)$ such that
  $d\omega=0$ and $\omega^n \neq 0$ at every point.
\end{definition}

\begin{definition} \label{def:resolution}
A {\it symplectic resolution} of a symplectic orbifold
$(M,\omega)$ is a smooth symplectic manifold
$(\tilde{M},\tilde{\omega})$ and a map $\pi:\tilde{M}\to M$ such
that:
  \begin{itemize}
  \item[(a)] $\pi$ is a diffeomorphism $\tilde{M}\setminus E \to M\setminus \Delta$,
  where $\Delta\subset M$ is the singular set and $E=\pi^{-1}(\Delta)$ is
  the {\it exceptional set}.
  \item[(b)] The exceptional set $E$ is a union of possibly intersecting smooth
  symplectic submanifolds of $\tilde{M}$ of codimension  at least 2.
  \item[(c)] $\tilde\omega$ and $\pi^*\omega$ agree in the complement of a
  small neighbourhood of $E$.
  \end{itemize}
\end{definition}

In \cite{NiPas}, it is given a method to obtain resolutions of
symplectic orbifolds arising as quotients pre-symplectic semi-free
$S^1$-actions. The following result gives an alternative method
which is valid for any symplectic orbifold, and which is inspired
in the resolution of isolated quotient singularities of complex
manifolds.

\begin{theorem} \label{theo:resolution}
  Any symplectic orbifold has a symplectic resolution.
\end{theorem}

\begin{proof}
Let $p$ be a singular point of $M$. Take an orbifold chart $U_p=
B_p/G_p$ around $p$, where $B_p\subset \RR^{2n}$ is a ball. The
symplectic form $\omega$ is a closed non-degenerate $2$-form
$\omega\in \Omega^2(U)$ invariant by $G_p$. By the equivariant
Darboux theorem, there is a symplectomorphism
$\varphi:(B_p,\omega) \to (B,\omega_0)\subset \RR^{2n}$, where $B$
is the standard ball and $\omega_0$ the standard symplectic form
in $\RR^{2n}$, and  a linear free $G_p$ action on
$\RR^{2n}\setminus\{0\}$ for which the map above is
$G_p$-equivariant (the proof of the existence of usual Darboux
coordinates in \cite[pp.\ 91-93]{McDuff-Salamon} carries over to
this case, only being careful that all the objects constructed
should be $G_p$-equivariant). Therefore, the orbifold admits
charts of the form $B/G_p$, where $B$ a symplectic ball of
$(\RR^{2n},\omega_0)$, such that $G_p$ acts linearly by
symplectomorphisms, that is $G_p\in Sp(2n,\RR)$.

Moreover, since the group $G_p \subset Sp(2n,\RR)$ is finite, we
may take a metric on $\RR^{2n}$ compatible with $\omega_0$ and
average it with respect to $G_p$. This gives a metric compatible
with $\omega_0$ and invariant by $G_p$ therefore producing a
$G_p$-invariant complex structure on $B$, so that we may interpret
$B\subset \CC^n$ and we have that $G_p\subset U(n)$. This induces
a complex structure $I$ on $B/G_p$, and
 $$
 B/G_p\subset \CC^n/G_p\,
 $$
has the complex and symplectic structure induced from the
canonical complex and symplectic structures of $\CC^n$.

For $n=1$, the only finite subgroups of $U(1)$ are cyclic groups
$\ZZ_m\subset U(1)$, and hence $B/G_p=\CC/\ZZ_m$ is already
non-singular. So in this case $M$ has already the structure of
smooth symplectic manifold.

For $n>1$, we work as follows. For each $p\in \Delta$ consider a
K\"ahler structure in a ball $U_p = B/G_p$ around $p$ as above,
where $B\subset \CC^n$. The singular complex variety $X=\CC^n/G_p$
is an affine algebraic variety with a single singularity at the
origin. We can take an algebraic resolution of the singularity (it
always can be done \cite{Hironaka} by successive blow-up along
smooth centers, starting with a single blowing-up at $p$), which
is a quasi-projective variety $\pi_X:\tilde X\to X$. The
exceptional set is a complex submanifold $E=\pi_X^{-1}(0)$.
Consider some embedding $\tilde X\subset \PP^N$ and let $\Omega$
be the induced K\"ahler form on $\tilde X$. Now let $\tilde{U}
=\pi^{-1}_X(U)$. We glue $\tilde{U}$ to $M\setminus \{p\}$ by
identifying $\tilde{U}\setminus E$ with $U_p\setminus \{p\}$ via
$\varphi^{-1}\circ\pi_X$, to get a smooth manifold
 $$
 \tilde{M}= (M\setminus \{p\}) \cup \tilde{U}\, .
 $$
There is an obvious projection $\pi:\tilde M\to M$.

We want to define a symplectic structure $\tilde\omega$ on
$\tilde{M}$ which equals $\omega$ on $M\setminus U_p$. Consider
the form $\varphi_*\omega$ on $U=B/G_p$ and its pull-back to
$\tilde{U}$ via $\pi_X$, $\omega'=\pi_X^*\varphi_*\omega$. The
annulus
 $$
 A= \left({\frac23}\bar B \setminus \frac13B\right) /G_p\subset B/G_p
 $$
is homotopy equivalent to $S^{2n-1}/G_p$, so we have that $\Omega-
\omega'=d\alpha$ on $A$, for some $\alpha\in \Omega^1(A)$. Let
$\rho$ be a bump function which equals zero in $(B\setminus
\frac23 B)/G_p$, and which equals one in $\frac13 B/G_p$. Define
  \begin{equation}\label{eqn:car2}
  \tilde\omega=\omega' + \epsilon \ d(\rho \alpha)\, .
  \end{equation}

Then $\tilde\omega= \omega'=\pi^*\omega$ on $(B\setminus \frac23
B)/G_p$, so it can be glued with $\omega$ to define a smooth
closed $2$-form on $\tilde{M}$. On $\frac13 B/G_p$, we have
  \begin{equation}\label{eqn:car}
  \tilde{\omega}= \omega'+ \epsilon (\Omega-\omega')=
   (1-\epsilon) \pi^*\omega +\epsilon \ \Omega \, .
  \end{equation}
Clearly, such $\tilde\omega$ is symplectic on $\frac13 B/G_p$
(actually it is a K\"ahler form there). Finally, as $A$ is
compact, the norm of  $d(\rho \alpha)$ on $A$ is bounded. As
$\omega'$ is symplectic on $A$, choosing $\epsilon>0$ small
enough, we get that $\tilde\omega$, defined in (\ref{eqn:car2}),
is also symplectic.
\end{proof}

Observe that for the symplectic resolution constructed in this
theorem, we can say more about the exceptional set since it is
modelled in the resolution of a  singularity on an algebraic
variety. Indeed, besides the conditions (a) -- (c) from Definition
\ref{def:resolution}, $\tilde{M}$ also satisfies
\begin{itemize}
 \item[(d)] There exists a complex structure $I$ on a \nhood\ of
 $E$ so that $(\tilde\omega,I)$ is a K\"ahler structure.
 \item[(e)] For the complex structure $I$ from (d) and $p\in
 \Delta$  one can find a complex structure in a \nhood\ $U$ of $p$
 making it K\"ahler and such that the resolution map
 $\pi:\tilde{U}\to U$ is holomorphic.
\end{itemize}

\subsection{Cohomology of resolutions}

We study the  singular cohomology of symplectic resolutions
with real coefficients.  For the symplectic resolution
$\tilde{M}$, this is isomorphic to the de Rham cohomology. For the
orbifold $M$, this is isomorphic to the orbifold de Rham
cohomology.

\begin{proposition}\label{prop:little-lemma}
  Let $\pi:\tilde{M}\to M$ be a symplectic resolution of a symplectic
  orbifold. For each $p\in \Delta$, let $E_p=\pi^{-1}(p)$ be the exceptional set.
  Then there is a split short exact sequence
  $$
   H^k(M) \stackrel{\pi^*}{\longrightarrow} H^k(\tilde M)
 \stackrel{i^*}{\longrightarrow}  \prod_{p\in \Delta} H^k(E_p) \,
 ,
   $$
   for $k>0$.
\end{proposition}

\begin{proof}
We may assume that $\Delta$ consists only of one point $p$, since
the general case follows from that by doing the resolution of the
singularities one by one and taking the inverse limit in the
noncompact case.

Let us see now that $i^*$ is surjective. For the singular point
$p$, let $U_p = B_p/G_p$ be a small ball around $p$ and
$\tilde{U}_p= \pi^{-1}(U_p)$ the corresponding \nhood\ of the
exceptional set $E_p$. Then $G_p$ acts freely and linearly on
$B_p\setminus \{p\}$. By choosing an invariant metric, we see that
$B_p$ is foliated by spheres invariant under the $G_p$ action, so
not only is $B_p\setminus \{p\}$ a deformation retract of the
rational homology sphere $S^{2n-1}/G_p$ but also $U_p =
(B_p\setminus\{p\})/G_p$ is a deformation retract of
$S^{2n-1}/G_p$. In particular $\tilde{U}_p\setminus E_p \cong
U_p\setminus\{p\}$ has the same real cohomology as $S^{2n-1}$.

Using the long exact sequence for relative cohomology for the pair
$(\tilde{U}_p,\del \tilde{U}_p)$, one easily sees that
$H^k(\tilde{U}_p,\del \tilde{U}_p)\cong H^k(\tilde{U}_p)\cong
H^k(E_p)$, for $0<k<2n-1$. For $k=2n-1$, we get
 $$
 0 \to H^{2n-1}(\tilde{U}_p,\del \tilde U_p) \to H^{2n-1}(\tilde U_p)=0
 \to H^{2n-1}(\bd \tilde U_p) =\RR \to H^{2n}(\tilde{U}_p,\del \tilde
 U_p) \to H^{2n}(\tilde U_p)=0 \, ,
 $$
so that $H^{2n-1}(\tilde{U}_p,\del \tilde U_p) =0$ and
$H^{2n}(\tilde{U}_p,\del \tilde U_p) =\RR$. Actually, as
$\tilde{U}_p$ is a compact oriented connected manifold with
boundary, $H^{2n}(\tilde{U}_p,\del \tilde U_p)$ is generated by
the fundamental class $[\tilde{U}_p,\del \tilde U_p]$. So we have
a map given as the composition
 $$
 f_p: H^k(E_p)\cong H^k(\tilde{U}_p) \cong H^k(\tilde{U}_p,\bd \tilde{U}_p)=
 H^k(\tilde{M},\tilde{M}\setminus \tilde{U}_p) \inc
 H^k(\tilde{M})\,,
$$
for $k>0$. It is easy to see that $i^*\circ f_p$ is the identity,
thus proving the surjectivity of $i^*$.

Now we prove that $\pi^*$ is injective. We define a map $\psi:
H^k(\tilde{M})\to H^k(M)$ for $0<k<2n-1$, as follows. The
Mayer-Vietoris exact sequence of $M=(M \setminus \{p\})\cup U_p$
gives
 $$
 \ldots \to H^{k-1}(S^{2n-1}/G_p) \to  H^k({M})
 \to H^k(M\setminus\{p\})\oplus H^k(U_p)\to H^{k}(S^{2n-1}/G_p)\to
 \ldots\, ,
 $$
so that $H^k(M)\cong H^k(M\setminus \{p\})$, for $0<k<2n-1$. We
define $\psi$ as the composition
 \begin{equation}\label{eq:psi:Mtilde into M}
 \psi:H^k(\tilde{M})  \into H^k(\tilde{M}\setminus E_p)
 \stackrel{\cong}{\into} H^k(M \setminus
 \{p \})\stackrel{\cong}{\into} H^k(M), \qquad k < 2n-1,
 \end{equation}
using that $\pi: \tilde{M}\setminus E_p \into M \setminus \{p\}$
is a diffeomorphism. Clearly, $\psi\circ \pi^*$ is the identity,
so that $\pi^*$ is injective for $0<k<2n-1$. For $k=2n-1$ and
$k=2n$, we have that $H^k(E_p)=0$, and there is a diagram whose
rows are exact sequences:
 $$
 \begin{array}{cccccc}
0 \to  H^{2n-1}({M}) \to& H^{2n-1}(M\setminus\{p\}) &\to
H^{2n-1}(S^{2n-1}/G_p) \to&  H^{2n}({M}) &\to
H^{2n}(M\setminus\{p\}) \to 0 \\
   \downarrow &\parallel & \parallel & \downarrow & \parallel
 \\
0\to H^{2n-1}(\tilde{M}) \to& H^{2n-1}(\tilde M\setminus E_p) &\to
H^{2n-1}(S^{2n-1}/G_p) \to& H^{2n}(\tilde{M}) &\to H^{2n}(\tilde
M\setminus E_p) \to 0\, ,
\end{array}
 $$
which proves the assertion.

It remains to see that the sequence is exact in the middle.
Clearly $\pi^*\circ i^*=0$. Also, for $k=2n-1,2n$ the statement is
clear, since the previous paragraph proves that in this case
$H^k(M)=H^k(\tilde M)$. For $0<k<2n-1$ we work as follows. We
write $\tilde{M}$ as a union of open sets, $\tilde{M} =
(M\setminus\{p\}) \cup \tilde{U}_p$, whose intersection
$(M\setminus\{p\}) \cap \tilde{U}_p = U_p\setminus\{p\}$ is
homotopic to $S^{2n-1}/G_p$. Since $\tilde{U_p}$ is a deformation
retract of $E_p$, the Mayer-Vietoris exact sequence gives
 $$
 \ldots \to  H^{k-1}(S^{2n-1}/G_p) \to  H^{k}(\tilde{M}) \to
 H^{k}(\tilde M\setminus E_p) \oplus H^k(E_p) \to
 H^{k}(S^{2n-1}/G_p)\to \ldots
 $$
For $0<k<2n-1$ we have isomorphisms $H^k(\tilde{M}) \cong
H^k(M\setminus \{p\})\oplus H^k(E_p) \cong H^k(M)\oplus H^k(E_p)$.
Actually, this map equals the map $(\psi, i^*)$, with $\psi$
defined in (\ref{eq:psi:Mtilde into M}). Hence the sequence is
exact in the middle.
\end{proof}

The last piece of data  we need to describe the product in
$H(\tilde{M})$ is the pairing
 $$
 H^k(\tilde{U}_p) \otimes H^{2n-k}(\tilde{U}_p,\bd \tilde{U}_p) \into H^{2n}(\tilde{U}_p,\bd
 \tilde{U}_p)\cong \RR\, ,
 $$
where the last isomorphism is given by integration on the
fundamental class $[\tilde{U}_p,\bd\tilde{U}_p]$. Combining this
pairing with the isomorphisms
 $$
 H^k(\tilde{U}_p,\bd \tilde{U}_p) \cong H^k(\tilde{U}_p)\cong H^k(E_p), \, 1\leq k \leq
 2n-1\,,
 $$
we have a map (the local intersection product),
 \begin{equation}\label{eq:local intersection}
 F_p: H^k(E_p) \otimes H^{2n-k}(E_p) \into \RR\,,
 \end{equation}
for $k=1,\ldots, 2n-1$.

\begin{proposition} \label{prop:local_product}
  Let $\pi:\tilde{M}\to M$ be a symplectic resolution of a compact connected symplectic
  orbifold, and
  let $\Psi:H^k(\tilde M)\to H^k(M)\oplus (\oplus H^k(E_p))$ be the isomorphism given by
  the split exact sequence in Proposition \ref{prop:little-lemma}, for $k>0$.
  Consider
  $a\in H^k(\tilde{M})$ and $b\in H^l(\tilde{M})$, with $k,l>0$, and denote
  $\Psi(a)=(a_1,(a_p)_{p\in \Delta})$ and
  $\Psi(b)=(b_1,(b_p)_{p\in \Delta})$. Then
   $$
   \Psi(a\cup b)= \left\{\begin{array}{ll}
  (a_1\cup b_1, (a_p\cup b_p)_{p\in\Delta}) , \qquad & k+l<2n\, ,
  \\[5mm]
  (a_1\cup b_1 + \sum\limits_{p\in\Delta} F_p(a_p,b_p), 0) , & k+l=2n\, .
 \end{array}\right.
   $$
\end{proposition}

\begin{proof}
As in the proof of Proposition \ref{prop:little-lemma}, it is
enough to do the case where there is only one singular point $p$.

Consider $a_1,b_1\in H^*(M)$, and let
$a=\Psi^{-1}(a_1,0)=\pi^*(a_1)$, $b=\Psi^{-1}(b_1,0)=\pi^*(b_1)$.
Then $\Psi^{-1}(a_1\cup b_1,0)= \pi^*(a_1\cup b_1)=a\cup b$, i.e.
$\Psi (a\cup b)=(a_1\cup b_1,0)$.

Now consider $a_1\in H^k(M)$, $b_p \in H^l(E_p)$, and let
$a=\Psi^{-1}(a_1,0)=\pi^*(a_1)$, $b=\Psi^{-1}(0,b_p)$. Then $b$
lies in the image of $H^l(\tilde{U}_p,\bd \tilde{U}_p)=
H^l(\tilde{M},\tilde{M}\setminus \tilde{U}_p) \inc
H^l(\tilde{M})$. So $a\cup b$ restricted to $\tilde{M} \setminus
\tilde U_p$ is zero, i.e., $\psi(a\cup b)=0$. On the other hand,
the restriction of $a$ to $E_p$ is zero, so $(a\cup b)|_{E_p}=0$.
Therefore $\Psi(a\cup b)=0$.

Finally, consider $a_p\in H^k(E_p)$, $b_p \in H^l(E_p)$, and let
$a=\Psi^{-1}(0,a_p)$, $b=\Psi^{-1}(0,b_p)$. Clearly, $(a\cup
b)|_{E_p}=a_p\cup b_p$. On the other hand, if $k+l<2n$, $\psi(a
\cup b)=0$, since $(a\cup b)|_{M\setminus U_p}=0$. If $k+l=2n$,
then $a\cup b$ is the image of $a_p\cup b_p$ under the map
 $$
 H^{2n}(\tilde{U}_p,\bd \tilde{U}_p)\to
 H^{2n}(\tilde{M},\tilde{M} \setminus \tilde{U}_p) \to H^{2n}(\tilde{M})=\RR\, .
 $$
Then $\Psi(a\cup b)=(F_p(a_p,b_p),0)$.
\end{proof}

\begin{remark}
The pairing $F_p$ is non-degenerate. We can prove this as follows:
take a compact orbifold $M$ with just one
singular point $p$
of the required type (see the proof of Theorem \ref{thm:3.9} where
a construction of such an orbifold is done). Then $\tilde M$ is a
compact oriented manifold, hence the intersection product
$H^k(\tilde M)\otimes H^{n-k}(\tilde M)\to \RR$ satisfies Poincar{\'e}
duality. By Proposition \ref{prop:local_product}, under the
isomorphism $\Psi:H^k(\tilde M)=H^k(M)\oplus H^k(E_p)$, the
intersection product of $H^k(\tilde M)$ decomposes as the
intersection product on $H^k(M)$ and the pairing $F_p$ on
$H^k(E_p)$. Hence both should be non-degenerate.

The non-degeneracy of $F_p$ implies that $\dim H^k(E_p) =\dim
H^{2n-k}(E_p)$. In particular, $H^1(E_p)=0$.
\end{remark}

\subsection{The Lefschetz property and resolutions}

Now we study how the Lefschetz property behaves under symplectic resolutions, and
prove that the resolution $(\tilde{M},\tilde\omega)$, constructed in
Theorem \ref{theo:resolution}, satisfies the Lefschetz
property \iff\ $(M,\omega)$ does.

Let $\pi:(\tilde{M},\tilde{\omega})\to (M,\omega)$ be a symplectic
resolution. Let $p$ be a singular point of $M$, then we have a
local intersection
 $$
 F_p: H^k(E_p)\otimes H^{2n-k}(E_p) \to \RR\, .
 $$

\begin{definition} \label{def:local-lefschetz}
 We say that the resolution satisfies the local Lefschetz property
 at $E_p$ if the map
 \begin{equation}\label{eqn:lefschetz}
 [\tilde\omega]^{n-k}:H^k(E_p)\to H^{2n-k}(E_p)
 \end{equation}
is an isomorphism for $k=1,\ldots, 2n-1$.
\end{definition}

Note that the above definition only depends on the restriction of
$[\tilde\omega]$ to $E_p$.

\begin{proposition} \label{prop:Lefs}
Let $\pi:(\tilde{M},\tilde\omega)\to (M,\omega)$ be a symplectic
resolution of a symplectic orbifold of dimension $2n$. Suppose
that $\pi$ satisfies the local Lefschetz property at every
divisor $E_p$, $p\in \Delta$. Then, for any $k=1,\ldots, n$, the
kernel of
 $$
 [\tilde\omega]^{n-k}: H^k(\tilde M) \into H^{2n-k}(\tilde M)\, .
 $$
is isomorphic to the  kernel of
 $$
 [\omega]^{n-k}: H^k(M) \into H^{2n-k}(M)\, .
 $$

In particular, if $(M,\omega)$ satisfies the Lefschetz property so does
$(\tilde{M},\tilde\omega)$.
\end{proposition}

\begin{proof}
We may suppose that we only do the resolution at one point. The
general case follows from this one. Also we may assume that $\dim
M=2n\geq 4$. By property (c) in Definition \ref{def:resolution},
$\tilde\omega$ and $\pi^*\omega$ agree on a neighbourhood of the
complement of the exceptional divisor, so denoting by
$\Psi:H^k(\tilde M)\to H^k(M)\oplus H^k(E_p)$ the isomorphism
coming from Proposition \ref{prop:little-lemma}, we have
 $$
 \Psi([\tilde\omega])=([\omega], [\tilde\omega|_{E_p}]).
 $$

By Proposition \ref{prop:local_product}, the map
 $$
 [\tilde\omega]^{n-k}: H^k(\tilde M) \into H^{2n-k}(\tilde M)\,
 $$
 decomposes under the isomorphism $\Psi$ as the direct sum of the
 two maps,
 $$
 [\omega]^{n-k}: H^k(M) \into H^{2n-k}(M)\, ,
 $$
and
 $$
 [\tilde\omega]^{n-k}:H^k(E_p)\to H^{2n-k}(E_p)\, .
 $$

If the local Lefschetz property is satisfied, the second map is an
isomorphism. The result follows.
\end{proof}

\begin{theorem} \label{thm:3.9}
The symplectic resolution $\pi:(\tilde M,\tilde \omega)\to
(M,\omega)$ constructed in Theorem \ref{theo:resolution} satisfies
the local Lefschetz property at $E_p$, for each singular point
$p\in M$.

So, $({M},\omega)$ satisfies the Lefschetz property if and only if  $(\tilde{M},\tilde\omega)$ does.
\end{theorem}

\begin{proof}
Take the complex projective variety $\bar X=\PP^n/G_p$ with the
linear action of $G_p$ which extend that on $\CC^n\subset\PP^n$.
Resolve its singularities \cite{Hironaka} at the infinity to
obtain a projective variety $Z$ with a single isolated singularity
at $p$. Let $\pi:\tilde{Z}\to Z$ be the resolution of the
singularity at $p$. Then $\tilde{Z}$ is a smooth projective
variety, hence it satisfies the hard-Lefschetz property, that is,
if $\Omega$ denotes the K\"ahler form of $\tilde Z$, then
 \begin{equation}\label{eqn:lll}
 [\Omega]^{n-k}: H^k(\tilde Z)\to H^{2n-k}(\tilde Z)
 \end{equation}
is an isomorphism.

Let $U=B/G_p$ be a small neighborhood of $p\in Z$, and let
$\tilde U=\pi^{-1}(U)$. Then Proposition \ref{prop:local_product}
implies that the map (\ref{eqn:lll}) decomposes as a direct sum of
the maps $[\omega_Z]^{n-k}:H^k(Z)\to H^{2n-k}(Z)$ and
$[\Omega]^{n-k}:H^k(E_p)\to H^{2n-k}(E_p)$, where
$E_p=\pi^{-1}(p)$ and $\omega_Z$ is the K\"ahler form of $Z$. So
the map
 \begin{equation}\label{eqn:lll2}
 [\Omega]^{n-k}: H^k(E_p)\to H^{2n-k}(E_p)
 \end{equation}
is an isomorphism.

Now let $\pi:(\tilde M,\tilde\omega)\to (M,\omega)$ be a
symplectic resolution at a point $p$ with local model $B/G_p$, as
carried out in Theorem \ref{theo:resolution}. Then
  $$
  \tilde\omega|_{\tilde{V}}=(1-\epsilon)\pi^*\omega + \epsilon\,
 \Omega\,,
 $$
in a neighborhood $\tilde V=\pi^{-1}(V)$ of $E_p$. But in $V$,
$\omega=d \gamma$ for a $1$-form $\gamma$, which we can suppose
$G_p$-invariant, so $\pi^*\omega$ is exact in $V$. So, restricting
to $\tilde{V}$, $[\tilde{\omega}]=[\epsilon \, \Omega]$. As the
map (\ref{eqn:lll2}) is an isomorphism, so is the map
 $$
 [\tilde\omega]^{n-k}: H^k(E_p)\to H^{2n-k}(E_p)\, ,
 $$
completing the theorem.
\end{proof}

\subsection{Resolutions and $a$-Massey products}

We show that $a$-Massey products are also well
behaved \wrt\ symplectic resolutions.

\begin{theorem} \label{theo:resolution-pseudo}
  Let $\pi:(\tilde{M},\tilde\omega)\to (M,\omega)$ be a symplectic resolution of a symplectic
  orbifold $(M,\omega)$. If $M$ has a non-trivial $a$-product, then
  so does $\tilde{M}$.
\end{theorem}

\begin{proof}
As before, we can assume, without loss of generality, that there
is only one singular point $p$.

Let $\A\subset \Omega_{orb}(M)$ be the algebra  of  smooth forms
which are constant (for degree $0$) and zero (for degree $>0$) in
a neighborhood of the  critical point $p$. Then the map
 $$
 \A \inc \Omega_{orb}(M)
 $$
is a quasi-isomorphism. According to  Lemma \ref{lem:quasi
isomorphisms}, there is a non-zero $a$-product
$\IP{a,b_1,\ldots,b_m}$ on $\A$. The inclusion
 $$
 \pi^*: \A \to \Omega(\tilde{M})\,
 $$
is a map of DGAs which induces the injection $\pi^*:H^k(M)\to
H^k(\tilde{M})$ for $k>0$ (it also induces an injection for
$k=0$). To prove our result we will show that
 $$
 \pi^*\la a;b_1,\ldots,b_m \ra\, = \la \pi^*(a);\pi^*(b_1),\ldots,\pi^*(b_m) \ra.
 $$
The inclusion
 $$
 \pi^*\la a;b_1,\ldots,b_m \ra\, \subset \la \pi^*(a);\pi^*(b_1),\ldots,\pi^*(b_m) \ra.
 $$
is obvious. So we only have to prove the converse.

Let $[c']=[\sum\overline{\xi_1'}\wedge \ldots \wedge
\pi^*(b_i)\wedge \ldots \wedge \xi_m'] \in \la
\pi^*(a);\pi^*(b_1),\ldots ,\pi^*(b_m) \ra$, where $\pi^*(a)\wedge
\pi^*(b_i)=d\xi_i'$. And let $\xi_i \in \A$ be such that $d\xi_i =
a \wedge b_i$. Then $d(\xi_i'-\pi^*(\xi_i))=0$, so
$\xi_i'-\pi^*(\xi_i)$ represents a cohomology class.


Using Proposition \ref{prop:little-lemma}, we may decompose
$[\xi_i'-\pi^*(\xi_i)] =\pi^*s_{i,1} + s_{i,p}$, where $s_{i,1}\in
H^k(M)$ and $s_{i,p} \in H(\tilde{U}_p,\bd\tilde{U}_p) \subset
H(\tilde M)$. Here we choose $U_p$ to be disjoint of the support
of $b_i$ for all $i$. We represent $s_{i,1}$ by a form $\zeta_i\in
\A\subset \Omega_{orb}(M)$ and $s_{i,p}$ by a form $\eta_{i}\in
\Omega(\tilde{U}_p,\bd\tilde{U}_p)$ (which can be thought of as a
form on $\tilde M$ supported inside $\tilde{U}_p$). So we can
write
 \begin{equation}\label{eq:xi, xi' zeta}
 \xi_i'-\pi^*(\xi_i)= \pi^*(\zeta_{i}) + \eta_{i} + dz_i,
 \end{equation}
where $z_i \in \Omega(\tilde{M})$. Therefore, $d(\xi_i +\zeta_i) =
a\wedge b_i$ and
 $$
 [c] = [\sum (\overline{\xi_1 +\zeta_1})\wedge \ldots \wedge b_i \wedge
 \ldots \wedge (\xi_m + \zeta_m)]  \in \IP{a;b_1,\ldots,b_m}
 $$
is such that
 \begin{align*}
 \pi^*[c]& = [\sum \overline{\pi^*(\xi_1 +\zeta_1)}\wedge \ldots
 \wedge \pi^*b_i \wedge \ldots \wedge \pi^*(\xi_m + \zeta_m)] \\
 &=[\sum (\overline{\xi_1' -\eta_{1})}\wedge \ldots \wedge
 \pi^*b_i \wedge \ldots \wedge (\xi_m' - \eta_m)]\\
 &=[\sum \overline{\xi_1'}\wedge \ldots \wedge \pi^*b_i
 \wedge \ldots \wedge \xi_m' ] \\
 &=[c'],
\end{align*}
where in the second equality we have used \eqref{eq:xi, xi' zeta},
Lemma \ref{lem:helps in the computations}  and Remark \ref{rem:eta
exact}, and in the third equality we used that $\eta_{i} \wedge
\pi^*b_j = 0$ since these forms have disjoint supports.  This shows
the reverse inclusion and finishes the theorem.
\end{proof}

\section{Symplectic blow-up}\label{sec:blow-up}

In this section we recall results about the behaviour of the
Lefschetz property under ordinary symplectic blow-up, as
introduced by McDuff \cite{McDuff}, and we study the behaviour
of $a$-products under this construction.

In what follows, we let $M^{2n}$ be a symplectic manifold/orbifold
and $N^{2(n-k)} \subset M$ be a symplectic submanifold which does
not intersect the orbifold singularities. We let $\pi:\tilde{M}
\into M$ be the symplectic blow-up of $M$ along $N$. Then the cohomology of
$\tilde{M}$ is given by
 $$
 H^i(\tilde{M}) = H^i(M)\oplus H^{i-2}(N)\sigma +
 H^{i-4}(N)[\sigma]^2 + \ldots + H^{i-2k+2}(N) [\sigma]^{k-1},
 $$
where $\sigma$ is a closed 2-form such that $\sigma^{k-1}$ has
nonzero integral over the $\CC P^{k-1}$ fibers of the exceptional
divisor. The multiplication rules are the obvious ones using the
restriction of elements on $H^i(M)$ to $H^i(N)$ together with the
extra relation
 $$
 [\sigma]^k = -PD(N)- c_{k-1}[\sigma] - \ldots - c_1[\sigma]^{k-1},
 $$
where $c_i$ are the Chern classes of the normal bundle of $N$, and
$PD(N)$ is the Poincar\'e dual of $N$.

Similarly to the case of symplectic resolutions the summand $H(M)
\hookrightarrow H(\tilde{M})$ is given by the image of the pull
back $\pi^*:H(M) \into H(\tilde{M})$. However, unlike the case of
resolutions, in general one can not choose representatives for the
cohomology classes in $H(M) \hookrightarrow H(\tilde{M})$ with
support away from the exceptional set.

\subsection{The Lefschetz property}

There is  a contrast between the behaviour of the maps
$[\omega]^{n-k}:H^k(M) \into H^{2n-k}(M)$ under resolution of
singularities and under ordinary symplectic blow-up. While we have
proved that in the former case these maps have the same kernel,
the same is not true for the latter. Indeed, in \cite{Cav2}, the
first author proved that one can reduce the dimension of the
kernel of the map $[\omega]^{n-k}$ by blowing-up along specific
submanifolds. The result from \cite{Cav2} adapted to the case we
study is the following:

\begin{theorem}[\cite{Cav2}]\label{theo:blow up and Lefschetz}
Given a symplectic orbifold $(M^{2n},\omega)$ and a symplectic
surface $\Sigma^2 \subset M$ disjoint from the singular set, let
$\pi:\tilde{M}\into M$  be the symplectic blow up of $M$ along
$\Sigma$. Then there is a symplectic form $\tilde{\omega}$ on
$\tilde{M}$ such that in $H^2(\tilde{M})$
 $$
 \ker([\tilde\omega]^{n-2}\cup) = \pi^*(\ker([\omega^{n-2}]\cup ) \cap \ker([\Sigma]\cap)),
 $$
while in $H^k(\tilde{M})$, for $k > 2$,
 $$
 \ker([\tilde\omega]^{n-k}\cup) =  \pi^*(\ker([\omega^{n-k}]\cup ).
 $$
\end{theorem}

\subsection{Symplectic blow-up and  $a$-Massey products}

Similarly, $a$-Massey products also behave differently under
symplectic blow-up. We focus our attention on the triple $a$-product.

\begin{theorem}\label{theo:blow up and G-products}
Let $(M^{2n},\omega)$ be a symplectic orbifold, $N^{2(n-k)}
\hookrightarrow M$ be a symplectic submanifold disjoint from the
orbifold singularities and $\tilde{M}$ the symplectic blow-up of
$M$ along $N$. Then
 \begin{enumerate}
 \item if $M$ has a nontrivial triple $a$-product, say, $\langle
 a;b_1,b_2,b_3\rangle$, and $|a|+ \frac{1}{2}(|b_i|+|b_j|) \leq
 k+1$, for all $i,j$, then so does $\tilde{M}$,
 \item if $H^{odd}(N)=\{0\}$, $k> 5$  and $N$ has a nontrivial triple $a$-product,
 then so does $\tilde{M}$.
 \end{enumerate}
\end{theorem}

\begin{proof}

We start with the proof of the first claim. Let $\langle a;
b_1, b_2, b_3 \rangle$ be a nontrivial $a$-product in
$M$. This means that $a \wedge b_i$ is exact  and
 $$
 0 \not \in \{[b_1\wedge \xi_2 \wedge \xi_3+ \overline{\xi_1}
 \wedge b_2 \wedge \xi_3 + \overline{\xi_1}\wedge \overline{\xi_2}\wedge
 b_3]\, |\, d\xi_i = a \wedge b_i \}.
 $$
Since the form $a \wedge b_i$ is exact  in $M$, $\pi^*a \wedge
\pi^*b_i $ is exact in $\tilde M$, hence the $a$-product is
defined on $\tilde{M}$. According to Lemma \ref{lem:helps in the
computations} and Remark \ref{rem:eta exact}, once we fix the
$\pi^*\xi_i$, the $a$-product is obtained by adding closed forms
to $\pi^*\xi_i$ and only depends on the cohomology class of the
closed forms added. In particular, we can assume that these closed
forms are of the standard form $\eta_i = \sum \eta_{ij} \sigma^j$,
so that the generic element of the product is given by
 $$
[\pi^*b_1\wedge (\pi^*\xi_2 + \eta_2) \wedge (\pi^*\xi_3+\eta_3)
 + (\overline{\pi^*\xi_1+\eta_1})\wedge \pi^*b_2 \wedge (\pi^*\xi_3
 + \eta_3) + (\overline{\pi^* \xi_1+ \eta_1})\wedge (\overline{\pi^*
 \xi_2 + \eta_2})\wedge \pi^*b_3].
 $$
Due to the hypothesis about $k$ and $|a|$, $|b_i|$, we see that
the component of the above class lying in $H(M) \subset
H(\tilde{M})$ is precisely  the original $a$-product as there are
no powers of $\sigma$ higher than $k-1$ appearing when the product
is computed. Since the original product was nontrivial, so is the
product induced  on $H(\tilde{M})$.

To prove the second claim, we let $\IP{a;b_1,b_2,b_3}$ be a
nontrivial $a$-product on $N$. This can only be the case if $a$
and $b_i$ are even degree forms, due to Lemma \ref{lem:cohomology
a b}, as $H^{odd}(N) = \{0\}$. Further, $H^{odd}(N) = \{0\}$
together with Proposition \ref{prop:indet} implies that
$\IP{a;b_1,b_2,b_3}$ has no indeterminacy and hence is a single
cohomology class. Now consider the closed forms $a\wedge \sigma,
b_i \wedge \sigma \in H^{even}(\tilde{M})$. The relations $a
\wedge b_i = d\xi_i$ imply that
 $$
 a\wedge \sigma \wedge b_i \wedge \sigma = d \xi_i \wedge \sigma^2,
 $$
and hence
 \begin{equation}\label{eq:representative}
 0\neq \IP{a;b_1,b_2,b_3}[\sigma]^5 \in \IP{a\wedge\sigma;b_1\wedge
 \sigma ,b_2 \wedge\sigma,b_3\wedge\sigma} \cap H(N)[\sigma]^5.
 \end{equation}
According to Proposition \ref{prop:indet}, the indeterminacy of
this product is a subset of
 $$
 \IP{b_1\wedge \sigma,a\wedge \sigma,b_2\wedge \sigma}
 H^{|b_3|-1}(\tilde{M})+  \IP{b_2\wedge \sigma,a\wedge
 \sigma,b_3\wedge \sigma} H^{|b_1|-1}(\tilde{M})+\IP{b_3\wedge
 \sigma,a\wedge \sigma,b_1\wedge \sigma}H^{|b_2|-1}(\tilde{M}).
 $$
Since $H^{odd}(N) = \{0\}$, all the triple Massey products above
lie in $H(M)  \subset H(\tilde{M})$ and also
$H^{|b_i|-1}(\tilde{M}) = H^{|b_i|-1}(M)$, so the indeterminacy of
the $(a\wedge\sigma)$-product is a subset of $H(M)$, but the
representative \eqref{eq:representative} does not belong to this
set, hence the product does not vanish.

\end{proof}

\begin{remark}
We must notice that for the case that we want to consider, that is, when $M$ is simply
connected and $8$-dimensional,
the hypothesis of the first
part of the
Theorem \ref{theo:blow up and G-products}
only hold if we are blowing up
along a symplectic submanifold of dimension 2 and $a$ and $b_i$
are forms of degree $2$. The second item was included for sake of
completeness and can only happen in higher dimensions. Indeed, the
first even-dimension where $a$-Massey products can appear is $8$,
hence in order for item 2 of the theorem above to be used one
should need the ambient manifold to be at least $20$-dimensional.
\end{remark}

\section{Examples}\label{examples}

%

In this section we  give an example of a simply connected
symplectic $8$-manifold which satisfies the Lefschetz property but
is not formal. In order to explain our example, we recall the
example given by the last two authors in \cite{FM4}.

\begin{example}[\cite{FM4}]
Consider $G$, the product of the complex Heisenberg group $H$,
with $\CC$. As a manifold $G$ is  diffeomorphic to $\CC^4$ but
with a group structure induced by the following embedding in
$\mathrm{GL}(5,\CC)$
 $$
 (z_1,z_2,z_3,z_4) \mapsto
 \begin{pmatrix} 1   &    z_1    &    z_3    &   0   &0\\
            0   &   1   &   z_2 &   0   &0\\
            0   &   0   &   1   &   0   &0\\
            0   &   0   &   0   &   1   &z_4\\
            0   &   0   &   0   &   0   &1\\
 \end{pmatrix}.
 $$
Letting $\xi$ be a cubic root of $1$, we have $1$ and $\xi$ generate a
lattice $\Lambda \subset \CC$ and then we obtain a cocompact
lattice  $\Gamma \subset G$ given by the matrices whose entries
lie in $\Lambda$. Further, the map
 $$
 \rho:G\into G, \qquad
 \rho(z_1,z_2,z_3,z_4) = (\xi z_1,\xi z_2,\xi^2 z_3,\xi z_4),
 $$
generates a $\ZZ_3$ action on $G$ which preserves  the lattice
$\Gamma$ and the group structure. Therefore it induces a $\ZZ_3$
action on the compact nilmanifold $\Gamma \setminus G$. This
action is free away from 81 fixed points corresponding to $z_i =
n/(1-\xi)$, for $n = 0,1$ and $2$.

The orbifold $M = \Gamma \setminus G /\ZZ_3$ has a symplectic
structure.  Indeed, if we consider the left-invariant complex 1-forms
$u^1 = dz_1, u^2 = dz_2, u^3 = dz_3 - z_1 dz_2, u^4 = dz_4$
defined on $G$, we see that  $du^1 = du^2= du^4 =0$, $du^3
=u^{12}$ and that
 $$
 \omega =  i u^{1\bar1} + u^{23} + u^{\bar2\bar3} + i u^{4\bar4}
 $$
is a $(\Gamma\times \ZZ_3)$-invariant symplectic 2-form, hence
induces a symplectic structure on the orbifold $M$, where we are
using the sort hand notation $u^{ij} = u^i \wedge u^j$, $u^{\bar
i}=\overline{u^i}$, $u^{i\bar j} = u^i \wedge \overline{u^j}$,
etc. The orbifold $M$ is simply connected and has vanishing odd
Betti numbers \cite{FM4}. Furthermore, it has a nonvanishing
$a$-Massey product. Indeed, if we let
 $$
 a = u^{1\bar1}, b_1 = u^{2\bar2}, b_2=u^{2\bar4}, b_3 = u^{\bar2 4},
 $$
then $a$ and $b_i$ are closed and invariant under the $\ZZ_3$
action, so define closed forms on $M$. Further
 $$
 a \wedge b_1 = -du^{12\bar3}, a\wedge b_2 = du^{\bar1 3 \bar4},
 a\wedge b_3 = -du^{1\bar3 4}.
 $$
Hence we can compute the $a$-Massey product
 \begin{align*}
 \IP{a;b_1,b_2,b_3} &= -u^{12\bar 3}\wedge u^{\bar1 3 \bar 4} \wedge
 u^{\bar2 4} - u^{\bar1 3 \bar4}\wedge u^{1\bar3 4}\wedge u^{2\bar2}
 + u^{1\bar3 4}\wedge u^{12\bar3}\wedge u^{2\bar4}\\
 &=2\, u^{1\bar1 2\bar2 3\bar 3 4 \bar4}.
\end{align*}
Since $H^5(M) = \{0\}$, the Massey products $\IP{b_i,a,b_j} \in
H^5(M)$ vanish and, according to Proposition \ref{prop:indet}, the
product above has indeterminacy zero, thus it is a non-trivial
$a$-Massey product. Finally, according to Theorems
\ref{theo:resolution} and \ref{theo:resolution-pseudo}, the
symplectic resolution of $M$ is a simply connected non-formal
8-dimensional symplectic manifold.
\end{example}

\begin{example}
As shown in \cite{FM4}, the orbifold $M$  obtained in the previous
example has vanishing odd Betti numbers, so in order to check
whether it satisfies the Lefschetz property, one only needs to
consider $[\omega]^2: H^2(M) \into H^6(M)$.
We show
that while for $M$ this map is not an isomorpism, one can blow $M$
up along three symplectic tori to obtain an orbifold which does
satisfy the Lefschetz property, but which still has non-trivial
$a$-Massey products.

We start determining the second cohomology of $M$. This is given
by the $\ZZ_3$-invariant part of the Lie algebra cohomology of
$G$ and has an ordered basis given by
 $$
 \{u^{1\bar1} ,u^{4\bar4} ,u^{23} ,u^{\bar2\bar3} ,u^{\bar1 2}
 ,u^{13} ,u^{1\bar2} ,u^{\bar1\bar3} ,u^{1\bar4} ,u^{\bar1 4}
 ,u^{2\bar2} ,u^{2\bar4} ,u^{\bar2 4} \},
 $$
where $u^i$ are the invariant 1-forms introduced in the previous
example. In this basis the pairing $[\omega]^2:H^{2}(M;\CC)\times
H^{2}(M;\CC) \into \CC$ is given by the following table

\bigskip

\setlength{\extrarowheight}{4pt}
\begin{tabular}{|c|c|c|c|c|c|c|c|c|c|c|c|c|c|}
 \hline
& $u^{1\bar1} $&$u^{4\bar4} $&$u^{23} $&$u^{\bar2\bar3}
$&$u^{\bar1 2} $&$u^{13} $&$u^{1\bar2} $&$u^{\bar1\bar3}
$&$u^{1\bar4}
$&$u^{\bar1 4} $&$u^{2\bar2} $&$u^{2\bar4} $&$u^{\bar2 4} $ \\
\hline $u^{1\bar1} $       &$\mathbf{0}$   &$\mathbf{-1}$
&$\mathbf{-i}$      &$\mathbf{-i}$      &0      &0      &0      &0
&0      &0      &0      &0      &0      \\ \hline $u^{4\bar4} $
&$\mathbf{1}$   &$\mathbf{0}$       &$\mathbf{-i}$ &$\mathbf{-i}$
&0      &0      &0      &0      &0      &0 &0      &0      &0
\\ \hline $u^{23}     $ &$\mathbf{-i}$  &$\mathbf{-i}$
&$\mathbf{0}$ &$\mathbf{1}$       &0      &0      &0      &0
&0      &0 &0      &0      &0      \\ \hline $u^{\bar2\bar3} $
&$\mathbf{-i}$  &$\mathbf{-i}$      &$\mathbf{1}$ &$\mathbf{0}$
&0      &0      &0      &0      &0      &0 &0      &0      &0
\\ \hline $u^{\bar1 2} $      &0      &0 &0      &0
&$\mathbf{0}$       &$\mathbf{-i}$      &0      &0 &0      &0
&0      &0      &0      \\ \hline $u^{13} $ &0      &0      &0
&0      &$\mathbf{-i}$      &$\mathbf{0}$ &0      &0      &0
&0      &0      &0      &0      \\ \hline $u^{1\bar2} $       &0
&0      &0      &0      &0      &0 &$\mathbf{0}$       &$\mathbf{
i}$      &0      &0      &0      &0 &0      \\ \hline
$u^{\bar1\bar3} $   &0      &0      &0      &0 &0      &0
&$\mathbf{ i}$      &$\mathbf{0}$       &0      &0 &0      &0
&0      \\ \hline $u^{1\bar4} $       &0      &0 &0      &0
&0      &0      &0      &0      &$\mathbf{0}$ &$\mathbf{-1}$
&0      &0      &0      \\ \hline $u^{\bar1 4}$       &0      &0
&0      &0      &0      &0      &0 &0      &$\mathbf{-1}$
&$\mathbf{0}$       &0      &0      &0
\\ \hline $u^{2\bar2} $       &0      &0      &0      &0      &0
&0      &0      &0      &0      &0      &\bf0       &0      &0
\\ \hline $u^{2\bar4} $       &0      &0      &0      &0      &0
&0      &0      &0      &0      &0      &0      &\bf0       &0
\\ \hline $u^{\bar2 4}$       &0      &0      &0      &0      &0
&0      &0      &0      &0      &0      &0      &0      &\bf0
\\ \hline
\end{tabular}

\bigskip

Hence, the kernel of $\omega^2$ has real basis $\{iu^{2\bar2}$,
$u^{2\bar4}+u^{\bar2 4}$, $i(u^{2\bar4}-u^{\bar2 4})\}$. Now we
split each $u^i$ into real and imaginary parts $u^j =
e^{2j-1}+ie^{2j}$, so that the kernel of $\omega^2$ is generated
by $e^{34}$, $e^{37}-e^{48}$ and $e^{47}+e^{38}$.

In terms of the real basis $\{e_i\}$, where $e_i$ is the invariant
vector field dual to $e^i$, the Lie algebra $\frak{g}$ of  $G$ has
the following structure:
 $$
 -[e_1,e_3]= - [e_2,e_4] = e_5, \qquad  -[e_1,e_4]= - [e_2,e_3] = e_6,
 $$
and the symplectic form is
 $$
 \omega = e^{12} + e^{35}-e^{46} +e^{78}.
 $$
Observe that since the lattice $\Gamma$ is given by matrices whose
entries are in the lattice $\Lambda$ generated by $1$ and $\xi$,
the vector fields $e_{2i-1}$ have period 1 ($1 \in \Lambda$),
while the vector fields $e_{2i}$ have period $\sqrt{3} =
\frac{1+2\xi}{i}$ (note that $1+2\xi \in \Lambda$).

For the example at hand, we consider the abelian Lie  subalgebras
of  $\frak{g}$ generated by
 $$
 \{e_3+e_7,e_4+e_8 \}\mbox{, }\{e_3 +\sqrt{3}\, e_8,e_7\} \
 \mbox{ and } \ \{e_3+e_7, e_8\}.
 $$
Each of these Lie algebras integrates to a Lie subgroup of $G$ and
the lattice $\Gamma$ restricts to a cocompact lattice on each of
the subgroups. Therefore, each of the abelian algebras gives rise
to a fibration of $\Gamma\setminus G$ by embedded tori. One can
clearly see that these tori are symplectic and by a general
position argument, we can choose three tori, $T_i$, one torus on
each family,  so that they do not intersect each other and also
they do not pass through the fixed points of the $\ZZ_3$ action.
 Thus, their image via the quotient map $\Gamma\setminus G
\into M$
are three disjoint embedded tori which do not meet the
orbifold singularities.

By Theorem \ref{theo:blow up and Lefschetz}, there is a
symplectic form $\tilde\omega$ on $\tilde{M}$,  the blow-up of $M$
along the three tori, such that the kernel of $[\tilde
\omega]^2:H^2(\tilde{M}) \into H^6(\tilde{M})$ is given by
 $$
 \pi^*(\ker([\omega^2]\cup)\cap \ker([T_1]\cap)\cap \ker([T_2] \cap)\cap \ker([T_3]\cap),
 $$
but by choice, each of the $[T_i]$ pairs nontrivially with one of
the elements in the basis $\{e^{34},
e^{37}-e^{48},e^{47}+e^{38}\}$ for $\ker([\omega^2]\cup)$. Hence
$\tilde \omega^2:H^2(\tilde{M}) \into H^6(\tilde{M})$ is an
isomorphism and the orbifold $\tilde{M}$ satisfies the Lefschetz
property. Further, taking account Theorem \ref{theo:blow up and
G-products}, $\tilde{M}$ has a nontrivial $a$-Massey product,
hence it is not formal.

According to Theorems \ref{theo:resolution}, \ref{thm:3.9} and
\ref{theo:resolution-pseudo}, the symplectic resolution of
$\tilde{M}$ satisfies the Lefschetz property and has a non-trivial
$a$-Massey product.
This example shows that the
Lefschetz property is not related to formality in dimension $8$.
\end{example}

{\small

\vspace{0.15cm}

\noindent{\sf G. R. Cavalcanti:} Mathematical Institute, St. Giles 24 -- 29, Oxford OX1 3LB, United Kingdom.

\noindent
{\it e-mail:} {\tt gil.cavalcanti@maths.ox.ac.uk}

\vspace{0.15cm}

\noindent{\sf M. Fern\'andez:} Departamento de Matem\'aticas,
Facultad de Ciencia y Tecnolog\'{\i}a, Universidad del Pa\'{\i}s
Vasco, Apartado 644, 48080 Bilbao, Spain.

\noindent
{\it e-mail:} {\tt marisa.fernandez@ehu.es}

\vspace{0.15cm}

\noindent{\sf V. Mu\~noz:} Departamento de Matem\'aticas, Consejo
Superior de Investigaciones Cient{\'\i}ficas, C/ Serrano 113bis, 28006
Madrid, Spain.

\noindent
\phantom{V. Mu\~noz:} Facultad de Matem{\'a}ticas, Universidad Complutense de Madrid, Plaza
de Ciencias 3, 28040 Madrid, Spain.

\noindent
{\it e-mail:} {\tt vicente.munoz@imaff.cfmac.csic.es}}

\end{document}